\documentclass[runningheads]{llncs}
\usepackage{amsmath}
\usepackage{graphicx}
\usepackage{hyperref}
\usepackage{tcolorbox}
\usepackage{colortbl}
\usepackage{geometry}
\usepackage[ruled,vlined,english,linesnumbered]{algorithm2e}
\newtheorem{thm}{Theorem}

\newtheorem{lem}[thm]{Lemma}
\begin{document}

\title{Complexity of tree-coloring interval graphs equitably\thanks{Supported by the National Natural Science Foundation of China (11871055, 11701440) and the Youth Talent Support Plan of Xi'an Association for Science and Technology (2018-6).}}
%
%
\author{Bei Niu\inst{1} \and
Bi Li\inst{1}\thanks{This author is the corresponding author.}
\and
Xin Zhang\inst{1}\thanks{This author shares a co-first authorship.}
}
\authorrunning{B. Niu, B. Li, and X. Zhang}
%
\institute{School of Mathematics and Statistics, Xidian University, Xi'an~710071,~China\\
\email{beiniu@stu.xidian.edu.cn, $\{$libi,xzhang$\}$@xidian.edu.cn}}
\maketitle              
\begin{abstract}
 An equitable tree-$k$-coloring of a graph is a vertex $k$-coloring such that each color class induces a forest and the size of any two color classes differ by at most one. In this work, we show that every interval graph $G$ has an equitable tree-$k$-coloring for any integer $k\geq \lceil(\Delta(G)+1)/2\rceil$, solving a conjecture of Wu, Zhang and Li (2013) for interval graphs, and furthermore, give a linear-time algorithm for determining whether a proper interval graph admits an equitable tree-$k$-coloring for a given integer $k$. For disjoint union of split graphs, or $K_{1,r}$-free interval graphs with $r\geq 4$, we prove that it is $W[1]$-hard to decide whether there is an equitable tree-$k$-coloring when parameterized by number of colors, or by treewidth, number of colors and maximum degree, respectively.

\keywords{$W[1]$-hardness \and linear-time algorithm \and equitable tree-coloring \and interval graph \and communication network.}
\end{abstract}

\section{Introduction}

A minimization model in graph theory so-called the equitable tree-coloring can be used to formulate a structure decomposition problem on the communication network with some security considerations \cite{zhang2019equitable}. Namely, an \emph{equitable tree-$k$-coloring} of a (finite, simple and undirected) graph $G$ is a mapping $c: V(G) \rightarrow \{1,2,\cdots,k\}$ so that $c^{-1}(i)$ induces a forest for each $1\leq i\leq k$, and $\big||c^{-1}(i)|-|c^{-1}(j)|\big|\leq 1$ for each pair of $1\leq i<j\leq k$. The notion of the equitable tree-$k$-coloring was introduced by Wu, Zhang and Li \cite{Wu20132696}, who conjectured that every graph $G$ has an equitable tree-$k$-coloring for any integer $k\geq \lceil(\Delta(G)+1)/2\rceil$. This conjecture (equitable vertex arboricity conjecture, EVAC for short) is known to have an affirmative answer in some cases including:
\begin{itemize}
    \item $G$ is complete or bipartite \cite{Wu20132696};
    \item $\Delta(G)\geq (|G|-1)/2$ \cite{Zhang2020581,Zhang2014217};
    \item $\Delta(G)\leq 3$ \cite{Zhang20161724};
    \item $G$ is 5-degenerate \cite{Chen2017426};
    \item $G$ is $d$-degenerate with $\Delta(G)\geq 10d$ \cite{zhang2019equitable};
    \item $G$ is IC-planar with $\Delta(G)\geq 14$ or $g(G)\geq 6$ \cite{Niu2020};
    \item  $G$ is a $d$-dimensional grid with $d\in \{2,3,4\}$ \cite{Drgas-Burchardt20186353}.
\end{itemize}
Anyway, EVAC is widely open.

Algorithmically, the following EQUITABLE TREE COLORING is NP-complete \cite{Gomes2019}.
\begin{tcolorbox}
EQUITABLE TREE COLORING\\
\noindent \textbf{Instance}: A graph $G$ and the number of colors $k$.\\
\noindent \textbf{Question}: Is there an equitable tree-$k$-coloring of $G$?
\end{tcolorbox}
\noindent Recently in \cite{Li2020156}, the last two authors proved that EQUITABLE TREE COLORING problem is W[1]-hard when parameterized by treewidth, and that it is polynomial solvable in the class of graphs with bounded treewidth, and in the class of graphs of bounded vertex cover number.

This paper focuses on interval graphs. A graph $G$ is an \emph{interval graph} if there exist an \emph{interval representation} of $G$, i.e., a family $\{T_v|v\in V(G)\}$ of intervals on the real line such that $u$ and $v$ are adjacent vertices in $G$ if and only if $T_u\cap T_v \not= \emptyset$. For any vertex $v\in V(G)$, $L(v)$ and $R(v)$ denote the left point and the right point of its corresponding interval $T_v$, respectively.
For two vertices $u,v\in V(G)$, if $L(u)<L(v)$, or $L(u)$=$L(v)$ and $R(u)\le R(v)$, then we write $u<v$.
For any three vertices $u,v,w\in V(G)$, it is clear that
\begin{align}\label{eq1}
 {\rm if} ~u<v<w ~and ~uw\in E(G), ~{\rm then}~ uv\in E(G).
\end{align}
Olariu \cite{Olariu199121} shows that a graph is an interval graph if and only if it has a \emph{linear order} $<$ on $V(G)$ satisfying (\ref{eq1}); and that the order can be found in linear time. Using this fact, we give the following result as a quick start of this paper, confirming EVAC for interval graphs.

\begin{thm}\label{interval}
Every interval graph $G$ has an equitable tree-$k$-coloring for any integer $k\geq \lceil\frac{\Delta(G)+1}{2}\rceil$, where the lower bound of $k$ is sharp.
\end{thm}

\begin{proof}
Sort the vertices of $G$ into $v_0<v_1<\cdots<v_{n-1}$ so that \eqref{eq1} holds, and for each $0\leq i\leq n-1$, let $c(v_i)=i~$(mod $k$). It is clear that $c$ is an equitable $k$-coloring of $G$. If there is a monochromatic cycle in color $i$, then there are three vertices $v_{i+\alpha k}$, $v_{i+\beta k}$ and $v_{i+\gamma k}$ with $0\leq \alpha<\beta<\gamma$ such that $v_{i+\alpha k} v_{i+\beta k}, v_{i+\alpha k} v_{i+\gamma k}\in E(G)$.
By \eqref{eq1}, $v_{i+\alpha k} v_{j}\in E(G)$ for any $i+\alpha k<j\leq i+\gamma k$, which implies $d_G(v_{i+\alpha k})\geq (\gamma-\alpha)k\geq 2k\geq \Delta(G)+1$, a contradiction. Hence, there is no monochromatic cycle under $c$. This implies that $c$ is an equitable tree-$k$-coloring of $G$. Since the complete graph $K_{2s}$ is an interval graph and it does not admit an equitable tree-$k$-coloring for any $k\leq s-1$, the lower bound of $k$ in this result is sharp.
\end{proof}

On the other hand, if we are given an integer $k<\lceil\frac{\Delta(G)+1}{2}\rceil$, determining whether an interval graph admits an equitable tree-$k$-coloring is not easy. Precisely, the next two theorems will be proved in Section \ref{sec:2}.

\begin{thm}\label{split}
{\rm EQUITABLE TREE COLORING} of the disjoint union of split graphs parameterized by number of colors is $W[1]$-hard.
\end{thm}


\begin{thm}\label{star-free}
{\rm EQUITABLE TREE COLORING} of $K_{1,r}$-free interval graph with $r\geq4$ parameterized by treewidth, number of colors and maximum degree is $W[1]$-hard.
\end{thm}

Here, a \emph{split graph} is a graph in which the vertices can be partitioned into a clique and an independent set, and a \emph{$K_{1,r}$-free graph} is a graph that does not contain the star $K_{1,r}$ as an induced subgraph.

However, the situation is much better if we are working with an \emph{proper interval graph}, that is an interval graph that has an interval representation in which no interval properly contains any other interval. Actually, we have the following theorem, which will be proved in Section \ref{sec:3}.

\begin{thm}\label{proper-interval}
There is a linear-time algorithm to determine whether a proper interval graph admits an equitable tree-$k$-coloring for a given integer $k$.
\end{thm}

 To end this section, we collect some notations that will be used in the next sections.
 For any two graphs $G$ and $H$, their {\it sum} $G\oplus H$ is the graph given by $V(G\oplus H) = V (G) \cup V (H)$ and $E(G \oplus H) = E(G) \cup E(H) \cup \{uv | u \in V (G),v \in V (H)\}$, and their {\it union} $G\cup H$ is the graph given by $V(G\cup H) = V (G) \cup V (H)$ and $E(G \cup H) = E(G) \cup E(H)$. By $nG$, we denote the $n$ disjoint copies of $G$, and $[k]$ stands for $\{1,2,\cdots,k\}$. Other undefined notations follow \cite{Diestelbook}.

\section{W[1]-hardness: the proofs of Theorems \ref{split} and \ref{star-free}}\label{sec:2}

All of our reductions involve the following BIN-PACKING problem, which is NP-hard in the strong sense \cite{GJ1979}, and is $W[1]$-hard when parameterized by the number of bins \cite{Jansen2010260,Jansen201339}.

\begin{tcolorbox}
BIN-PACKING\\
\noindent \textbf{Instance}: A set of $n$ items $A=\{a_1, a_2, ..., a_n\}$ and a bin capacity $B$.\\
\noindent \textbf{Parameter}: The number of bins $k$.\\
\noindent \textbf{Question}: Is there a $k$-partition $\varphi$ of $A$ such that, $\forall$ $i \in [k]$, $\sum_{a_j \in \varphi_i} a_j = B$?
\end{tcolorbox}

\noindent {\bf \emph{Proof of Theorem \ref{split}}}.
Given an instance of BIN-PACKING  as above, Our strategy is to construct a disjoint union of split graph $G$ such that the answer of the BIN-PACKING is YES if and only if $G$ admits an equitable tree-$k$-coloring.
Here, $$G=\bigcup_{j\in[n]}H(a_j,k),$$ where $a_1,a_2,\cdots,a_n$ are arbitrarily given integers in the instance of BIN PACKING, and $$H(a,k)=K_{2k-1}\oplus (a+1)K_1$$ defines a split graph for integers $a$ and $k$. For $j\in [n]$, let $I_j$ be the independent set of size $a_j+1$ in $H(a_j, k)$, and let $c_j$ be a fixed vertex in the clique part $K_{2k-1}$ of $H(a_j, k)$.

Suppose that there is a $k$-partition $\varphi$ of $A=\{a_1, a_2, ..., a_n\}$ such that, $\forall$ $i \in [k]$, $\sum_{a_j \in \varphi_i} a_j = B$. For any $i \in [k]$ and for any $j$ satisfying $a_j \in \varphi_i$, color $c_j$ and all vertices in $I_j$ with color $i$. For any $j\in [n]$ such that $a_j\in \varphi_i$, color the $2k-2$ vertices in $V(H(a_j,k))\setminus (I_j\cup \{c_j\})$ with $k-1$ distinct colors in $[k]\setminus i$ so that each color is used exactly twice. At this moment, for any $j\in [n]$, $H(a_j, k)$ has been colored with $k$ colors so that the set of all vertices with color $i$ (here $i$ is the integer such that $a_j \in \varphi_i$) induces a forest, which is actually a star with center $c_j$, and each of another $k-1$ colors besides $i$ is used for exactly two vertices. Hence this gives a tree-$k$-coloring of $H(a_j, k)$ for any $j\in [n]$, and thus finally gives a tree-$k$-coloring of $G$.
To see that this coloring is an equitable tree-$k$-coloring of $G$, we denote the set of vertices with color $i$ as $V_i$ for any color $i\in [k]$. Clearly, $$|V_i|=\sum_{a_j\in \varphi_i} (a_j+2)+\sum_{a_j\not\in \varphi_i} 2=\sum_{a_j\in \varphi_i} a_j+2n=B+2n $$ for any $i\in[k]$, which implies that such a coloring is equitable.

On the other direction, if $G$ admits an equitable tree-$k$-coloring $\psi$, then in the clique part $K_{2k-1}$ of each $H(a_j, k)$ with $j\in [n]$, there is a color $i$ appearing on exactly one vertex, and each of another $k-1$ colors appears on exactly two vertices. It follows that all vertices in $I_j$ of $H(a_j, k)$ are colored with $i$, since any other color classes contains two vertices in the clique part $K_{2k-1}$, which are adjacent to all vertices in $I_j$. Therefore, taking any one vertex $v_j\in I_j$ together with  the clique part $K_{2k-1}$ induces a clique $K_{2k}$ containing exactly two vertices in each color class. 
Let $\psi_i$ be the vertices of $G$ colored with $i$ under the coloring  $\psi$.
We show that $\sum_{I_j\subseteq\psi_i} a_j=B$, which indicates that the answer for the BIN-PACKING is YES. 

Since there are $$n(2k-1)+\sum_{j\in [n]} (a_j+1)=k(2n+B)$$ vertices in $G$ (note that in BIN-PACKING we always assume that $\sum_{j\in [n]} a_j=kB$), each color class of $\psi$ contains exactly $2n+B$ vertices, which consists of, for each $j\in [n]$, two vertices in the clique $K_{2k}$ of $H(a_j,k)$ as chosen above and $a_j$ vertices in $I_j\backslash \{v_j\}$ if $I_j\subseteq \psi_i$. 
So $$2n+B=|\psi_i|=2n+\sum_{I_j\subseteq \psi_i} a_j,$$ which gives $\sum_{I_j\subseteq\psi_i} a_j=B$.

\begin{lem}\label{inteval}\cite{Fishburn1985,McMorris}
  A graph is an interval graph if and only if its maximal cliques can be ordered as $M_1,M_2,\cdots,M_k$ such that for any $v\in M_i\cap M_k$ with $i < k$, $v\in M_j$ for any every $i\leq j\leq k$.
\end{lem}

\noindent {\bf \emph{Proof of Theorem \ref{star-free}}}.
We prove the theorem for parameter number of colors; and the theorem for another two parameters can be proved in the same way. Given an instance of the BIN-PACKING, our strategy is to construct a $K_{1,r}$-free interval graph $G$ with $r\geq4$ such that the answer of the BIN-PACKING is YES if and only if $G$ admits an equitable tree-$k$-coloring.
Here, $$G=\bigcup_{j\in[n]}J(a_j,k)$$ with $a_1,a_2,\cdots,a_n$ being arbitrarily given integers in the instance of BIN PACKING, and 
$$J(a,k)=\bigg(\bigcup_{i\in[a]}\big(Q_i\oplus y_i\big)\bigg) \bigcup \bigg(\bigcup_{i\in[a]}\big(Q'_i\oplus y_i\big)\bigg) \bigcup \bigg(\bigcup_{i\in[a-1]}\big(Q_{i+1}\oplus y_i\big)\bigg),$$ 
where
$S=\{Q_1,Q'_1,\cdots,Q_a,Q'_a\}$ is a set of cliques such that $Q_i\simeq Q'_i\simeq K_{2k-1}$ and $Y=\{y_1,\cdots,y_a\}$ is a set of vertices. Note that the vertices of $G$ with largest degree are the ones contained in $Y\setminus \{y_a\}$, which have degrees equal to $3(2k -1)$, and the treewidth of $G$ is $2k-1$.

We claim first that $G$ is an interval graph. Indeed, it is sufficient to show that $J(a,k)$ is an interval graph for any positive  integer $a$. By the definition of $J(a,k)$, one can see that it has $3a-1$ maximal cliques $M_1,M_2,\cdots,M_{3a-1}$ such that
\begin{align*}
        M_i
        &=
        \begin{cases}
        Q_i\oplus y_i
        &\text{if }i\equiv 1 ~(\text{mod }3)
        \\
        Q'_i\oplus y_i
        &\text{if }i\equiv 2 ~(\text{mod }3)
        \\
        Q_{i+1}\oplus y_i
        &\text{if }i\equiv 0 ~(\text{mod }3).
        \end{cases}
   \end{align*}
Since $M_i\cap M_{j}=\emptyset$ for any $i\equiv 1 ~(\text{mod }3)$ and $j\geq i+3$,
$M_i\cap M_{j}=\emptyset$ for any $i\not\equiv 1 ~(\text{mod }3)$ and $j\geq i+2$,
and $M_i\cap M_{i+1}=M_i\cap M_{i+2}=M_{i+1}\cap M_{i+2}=\{y_i\}$ for any $1\leq i\leq 3a-5$ with $i\equiv 1 ~(\text{mod }3)$, the ordering $M_1,M_2,\cdots,M_{3a-1}$ satisfies the property described by Lemma
\ref{inteval}, and therefore, $J(a,k)$ is an interval graph.

Suppose that there is a $k$-partition $\varphi$ of $A$ such that, $\forall$ $i \in [k]$, $\sum_{a_j \in \varphi_i} a_j = B$. 
For any $i \in [k]$ and for any $j$ satisfying $a_j \in \varphi_i$, color all vertices of $Y_j$, corresponding to the vertex set $Y$ in $J(a_j,k)$, with color $i$.  
For any $j\in [n]$ such that $a_j\in \varphi_i$, color each $(k-1)$-clique of $S_j$, corresponding to the set $S$ of cliques in $J(a_j,k)$, so that the color $i$ is used for exactly one vertex, and each of the remaining $k-1$ colors in $[k]\backslash i$ are used for exactly two vertices.
Clearly, this gives a tree-$k$-coloring of $J(a_j, k)$, where $3a_j$ vertices consisting of $Y_j$ and one vertex in each clique in $S_j$, are colored with color $i$, and each of another $k-1$ colors is used for exactly two vertices in each clique of $S_j$. Hence a tree-$k$-coloring of $G$ is given now. To see that this is an equitable tree-$k$-coloring of $G$, we denote the set of vertices with color $i$ as $V_i$ for any color $i\in [k]$. The fact that  $$|V_i|=\sum_{a_j\in \varphi_i} 3a_j+\sum_{a_j\notin \varphi_i} 4a_j=3B+ 4(kB-B)=(4k-1)B$$ for any $i\in [k]$ implies the equability of this coloring.

On the other direction, if $G$ admits an equitable tree-$k$-coloring $\psi$, then in each clique of $S_j$ for $j\in [n]$,
there is a color appearing on exactly one vertex, and each of another $k-1$ colors appears on exactly two vertices. Suppose that the color $i$ appears on exactly one vertex of the first clique $Q_1\in S_j$ for some $j\in [n]$. It follows that $y_1$ should be colored with $i$ because $y_1$ is adjacent to every vertices of $Q_1$ and each color in $[k]\backslash i$ already appears on two vertices of $Q_1$. Consequently, the color being used exactly once for the vertices of the second clique $Q'_1\in S_j$ and the third clique $Q_2\in S_j$ is indeed $i$, which implies that $y_2$ shall be colored with $i$. Following this process, we can conclude that each vertex of $Y_j$ is colored with $i$, and in each clique of $S_j$,
the color $i$ appears on exactly one vertex, and each of another $k-1$ colors in $[k]\backslash i$ appears on exactly two vertices.  Let $\psi_i$ be the vertices of $G$ colored with $i$ under the coloring  $\psi$.
We show that $\sum_{Y_j\subseteq\psi_i} a_j=B$, which indicates that the answer for the BIN-PACKING is YES. 

Since there are $$\sum_{j\in [n]} \big(a_j+2a_j(2k-1)\big)=kB+(4k-2)kB=k(4k-1)B$$ vertices in $G$, for each $i\in [k]$, $$|\psi_i|=(4k-1)B,$$ and by the way of the coloring $\psi$ as described above, we also see that
$$|\psi_i|=\sum_{Y_j\subseteq \psi_i} 3a_j+\sum_{Y_j\cap \psi_i=\emptyset} 4a_j = \sum_{j\in [n]} 4a_j-\sum_{Y_j\subseteq  \psi_i} a_j=4kB-\sum_{Y_j\subseteq  \psi_i}a_j.$$
Combining the two expression gives $\sum_{Y_j\subseteq\psi_i} a_j=B$.

\section{Linear-time algorithm: the proof of Theorem \ref{proper-interval}}\label{sec:3}

\begin{lem}\label{clique-induced}
Let $G$ be a proper interval graph with $V(G)=\{v_1,v_2,\cdots,v_n\}$, where $v_1<v_2<\cdots<v_{n}$. If $v_iv_j\in E(G)$, 
then  $\{v_i,v_{i+1},\cdots,v_{j-1},v_j\}$ induces a clique of size $j-i+1.$
\end{lem}

\begin{proof}
By the definition of the proper interval graph, for any $i\leq s<\ell\leq j$, $L(v_s)< L(v_{\ell})$, and if $v_iv_j\in E(G)$, then $L(v_i)\leq L(v_s)<L(v_{\ell})\leq L(v_j)\leq R(v_i)\leq R(v_s)<R(v_{\ell})\leq R(v_j)$. This implies that the interval $[L(v_s),R(v_s)]$ intersects the interval $[L(v_{\ell}),R(v_\ell)]$ and thus $v_kv_{\ell}\in E(G)$.
Hence any two vertices among $\{v_i,v_{i+1},\cdots,v_{j-1},v_j\}$ are adjacent and such a  vertex set induces a clique.
\end{proof}

\begin{lem}\label{lem5}
  A proper interval graph has an equitable tree-$k$-coloring if and only if its maximum clique has size at most $2k$.
\end{lem}

\begin{proof}
If $c$ is an equitable tree-$k$-coloring of $G$, then there is no clique on at least $2k+1$ vertices, because otherwise there is a color appearing at least three times on this clique, implying the existence of a monochromatic triangle, a contradiction. Hence the maximum clique of $G$ has size at most $2k$.

On the other direction, if the maximum clique of $G$ has size at most $2k$, then sort the vertices of $G$ into $v_0,v_1,\cdots,v_{n-1}$ so that $v_i<v_j$ if $i<j$. Let $c(v_i)=i~$(mod $k$). It is clear that $c$ is an equitable $k$-coloring of $G$. If there is a monochromatic cycle under $c$, then there are two adjacent vertices $v_i$ and $v_j$ with $j-i=\beta k$ with $\beta\geq 2$. By Lemma \ref{clique-induced}, $G$ contains a clique of size $j-i+1=\beta k+1\geq 2k+1$ as a subgraph, a contradiction. This implies that there is no monochromatic cycle under $c$, and thus $c$ is an equitable tree-$k$-coloring of $G$.
\end{proof}

\begin{algorithm}\label{alg:1}
\BlankLine
\KwIn{A proper interval graph $G=(V,E)$ on $n$ vertices; a set of $k$ colors $\{0,1,\cdots,k-1\}$;}
\KwOut{$Answ$;}
\BlankLine
$Answ$ $\leftarrow$ $YES$;\\
Sort the vertices of $G$ into $v_0<v_1<\cdots<v_{n-1}$, where $<$ is a linear order on $V(G)$;\\
\For{$i=0$ to $n-1$}{
  Color vertex $v_i$ with the color $c(i)=i$ (mod $k$);
  }
\If{there is a monochromatic cycle in any color class}{
  Output $NO$;}
Output $Answ$.
\caption{Linear-time Algorithm for proper interval graphs}
\label{alg: properinterval}
\end{algorithm}

\begin{thm}\label{thm4}
Given a proper interval graph $G$ and an integer $k>0$, Algorithm ~\ref{alg: properinterval} outputs YES in linear time if and only if there exists an equitable tree-$k$-coloring of $G$; moreover, if YES, it gives an equitable tree-$k$-coloring of $G$.
\end{thm}
\begin{proof}

From Algorithm \ref{alg:1}, one sees that the given coloring is an equitable tree-$k$-coloring of $G$ if it outputs YES. If there is an equitable tree-$k$-coloring of $G$, then the size of the maximum clique of $G$ is at most $2k$ by Lemma \ref{lem5}. In any iterative step of Algorithm \ref{alg:1}, every color class induces disjoint unions of paths, because if not, there are two adjacent vertices $v_sv_\ell$ with $\ell-s\ge 2k$, which gives a clique of size $2k+1$ by Lemma \ref{clique-induced}, a contradiction. So the algorithm outputs YES. The time complexity dominates by Line 3, which takes $O(|V|+|E|)$ time, see \cite[Theorem 6]{Olariu199121}.  
\end{proof}

\noindent
{\bf \emph{Proof of Theorem \ref{proper-interval}}}.
This is an immediate corollary from Theorem \ref{thm4}.
 \vspace{2mm}

\noindent {\bf Remark.} Lemma \ref{lem5} implies that {\rm EQUITABLE TREE COLORING} of proper interval graphs is equivalent to determine whether $2k$ is the upper bound of its clique number. We know that to calculate all maximal cliques of a triangulated graph  $G=(V,E)$ (i.e, a graph without induced cycles on at least four vertices)  can be done in $O(|V|+|E|)$ time \cite[Theorem 4.17]{GOLUMBIC198081}, and any proper interval graph is a triangulated graph by Lemma \ref{clique-induced}. This also proves Theorem \ref{proper-interval}, however, without giving an equitable tree-$k$-coloring of $G$ if the algorithm outputs YES.

\section*{Acknowledgements}

The last author would like to acknowledge the supports provided by China Scholarship Council (CSC) under the grant number 201906965003 and by Institute for Basic Science (IBS, South Korea) during a visit of him to Discrete Mathematics Group, IBS.
%
%
%
%

\end{document}